\documentclass[a4paper,11pt,leqno]{amsart}
\usepackage[cp1250]{inputenc}\usepackage{ifthen}\usepackage{amsmath,amsthm,amssymb,amstext,amsfonts,amscd}
\usepackage{latexsym}

0.3cm \numberwithin{equation}{section} \textwidth 15cm \textheight
23cm
\newtheorem{theorem}{Theorem}[section]
\theoremstyle{plain}

\newtheorem{corollary}[theorem]{Corollary}

\newtheorem{remark}[theorem]{Remark}

\numberwithin{equation}{section}

\begin{document}
\title[Univalence  and quasiconformal extensions of an operator]{Univalence and quasiconformal extension of an integral
operator}
\author{\ Erhan Deniz}
\address{(E. Deniz) \textit{Department of Mathematics$,$ Faculty of Science}
\textit{and Letters,} \textit{Kafkas University$,$ Kars$,$ Turkey} }
\email{\textbf{edeniz36@gmail.com}}
\author{\ Stanisława Kanas}
\address{(S. Kanas) Institute of Mathematics, University of Rzeszów, Rzeszó%
w, Poland}
\email{\textbf{skanas@ur.edu.pl}}
\author{\ Halit Orhan}
\address{(H. Orhan) \textit{Department of Mathematics$,$ Faculty of Science}
\textit{Atatürk University$,$ TR-$25240$ Erzurum$,$ Turkey}}
\email{\textbf{horhan@atauni.edu.tr}} 

\begin{abstract}
In this paper we give some sufficient conditions of analyticity  and
univalence for functions defined by an integral operator.  Next, we
refine the result to a quasiconformal extension criterion with the
help of the Becker's method. Further,  new univalence criteria  and
the significant relationships with other results are given. A number
of known univalence conditions would follow upon specializing the
parameters involved in our main results.
\end{abstract}

\maketitle

\section{\textbf{Introduction}}

Denote by $\mathcal{U}_{r}=\left\{ {z\in \mathbb{C}:\;\left\vert
z\right\vert <r}\right\} \quad (0<r\le 1)$ the disk of radius $r$
and let $\mathcal{U}=\mathcal{U}_{1}.$ Let $\mathcal{A}$ denote the
class of \textit{analytic functions} in the open unit disk
$\mathcal{U}$ which satisfy the
usual normalization condition $f(0)=f'(0)-1=0$, and let $%
\mathcal{S}$ be the subclass of $\mathcal{A}$ consisting of the
functions $f$ which are \textit{univalent} in $\mathcal{U}.$ Also,
let $\mathcal{P}$ denote the class of functions
$p(z)=1+\sum_{n=1}^{\infty }p_{k}z^{k}$ that satisfy the condition
$\Re\ p(z)>0$ ($z\in \mathcal{U}$), and $\Omega$ be a class of
functions $w$ which are analytic in $\mathcal{U}$ and such that
$\left\vert w(z)\right\vert <1$ for $z\in \mathcal{U}.$ These
classes have been one of the important
subjects of research in geometric function theory for a long time (see \cite%
{SrOw}).

We say that a sense-preserving homeomorphism $f$ of a plane domain
$G\subset \mathbb{C}$ is $k-$\textit{quasi- conformal}, if $f$ is
absolutely continuous on almost all lines parallel to coordinate
axes and $\left\vert f_{\overline{z}}\right\vert \leq
k\left\vert f_{z}\right\vert$, almost everywhere in $G,$ where $f_{\overline{z}%
}=\partial f/ \partial \overline{z},$ $f_{z}=\partial f/
\partial z$ and $k$ is a real constant with $0\leq k<1.$
For the general defnition of quasiconformal mappings see \cite{Ah}.

Univalence of complex functions is an important property but, in
many cases  is impossible to show directly that a certain function
is univalent. For this reason, many authors found different
sufficient conditions of univalence. Two of the most important  are
the well-known criteria of Becker \cite{Be} and Ahlfors \cite{Ah}.
Becker and Ahlfors' works depend upon a ingenious use of the theory
of the Loewner chains and the generalized Loewner differential
equation. Extensions of these two criteria were given by Ruscheweyh
\cite{Ru}, Singh and Chichra \cite{SiCh},  Kanas and Lecko
\cite{KL1, KL2} and Lewandowski \cite {Le}. The recent
investigations on this subject are due to Raducanu et\ al.
\cite{RaOrDe} and Deniz and Orhan \cite{DeOr}, \cite{DeOr1}.
Furthermore, Pascu \cite{Pa} and Pescar \cite{Pe} obtained some
extensions of Becker's univalence criterion for an integral
operator, while Ovesea \cite{Ov1} obtained a generalization of
Ruscheweyh's univalence criterion for an integral operator.

In the present paper, we  formulate a new criteria for univalence of
the functions defined by an integral operator $G_{\alpha }$,
considered in \cite{Ov1}, and improve obtained there results. Also
we obtain a refinement to a quasiconformal extension criterion of
the main result. In the special cases, our univalence conditions
contain the results obtained by some of the authors cited in
references. Our considerations are based on the theory of Loewner
chains.

\section{\textbf{Loewner chains and quasiconformal extension}}

The method of Loewner chains will prove to be crucial in our later
consideration therefore we present a brief summary of that method.

Let $\mathcal{L}(z,t)=a_{1}(t)z+a_{2}(t)z^{2}+...,$ $a_{1}(t)\neq 0,$ be a
function defined on $\mathcal{U}\times I$, where $I:=[0,\infty )$ and $%
a_{1}(t)$ is a complex-valued, locally absolutely continuous
function on $I.$ Then $\mathcal{L}(z,t)$ is said to be
\textit{Loewner chain} if $\mathcal{L}(z,t)$ has the following
conditions;

\begin{enumerate}
\item[(i)] $\mathcal{L}(z,t)$ is analytic and univalent in $\mathcal{U}$ for
all $t\in I$

\item[(ii)] $\mathcal{L}(z,t)\prec \mathcal{L}(z,s)$ for all $0\leq t\leq
s<\infty $,
\end{enumerate}
\noindent where the symbol \textquotedblright $\prec $
\textquotedblright\ stands for subordination. If $a_{1}(t)=e^{t}$
then we say that $\mathcal{L}(z,t)$ is a \textit{standard Loewner
chain}.

In order to prove our main results we need the following theorem due
to Pommerenke \cite{Po1} (see also \cite{Po}). This theorem is often
used to find out univalency for an analytic function, apart from the
theory of Loewner chains.

\begin{theorem}
\label{T1} \cite[Pommerenke]{Po}\textbf{\ }\textit{Let }$\mathcal{L}%
(z,t)=a_{1}(t)z+a_{2}(t)z^{2}+...$\textit{\ be analytic in }$\mathcal{U}_{r}$%
\textit{\ for all }$t\in I.$ Suppose that;

\begin{enumerate}
\item[(i)] $\mathcal{L}(z,t)$ is a \textit{locally absolutely continuous
function in\ the interval }$I,$\textit{\ and locally uniformly with respect
to }$\mathcal{U}_{r}.$\textit{\ }

\item[(ii)] $a_{1}(t)$ is a complex valued continuous function on $I$ such
that $a_{1}(t)\neq 0,$ $\left\vert {a_{1}(t)}\right\vert \rightarrow \infty $%
\textit{\ for }$t\rightarrow \infty $ and%
\begin{equation*}
\left\{ \frac{{\mathcal{L}(z,t)}}{a_{1}(t)}\right\} _{t\in I}
\end{equation*}%
\textit{\ forms a normal family of functions in }$\mathcal{U}_{r}.$

\item[(iii)] There exists an analytic function $p:\mathcal{U}\times
I\rightarrow \mathbb{C}$ satisfying $\Re {p(z,t)}>0$\textit{\ for
all }$z\in \mathcal{U},\;t\in I$ and
\begin{equation}
z\frac{\partial \mathcal{L}(z,t)}{\partial z}=p(z,t)\frac{\partial \mathcal{L%
}(z,t)}{\partial t}\quad (z\in \mathcal{U}_{r},\ t\in I). \label{1}
\end{equation}%
\textit{Then, for each }$t\in I,$\textit{\ the function }$\mathcal{L}(z,t)$%
\textit{\ has an analytic and univalent extension to the whole disk }$%
\mathcal{U}$ or the function $\mathcal{L}(z,t)$ is a Loewner chain.
\end{enumerate}
\end{theorem}

The equation (\ref{1}) is called the \textit{generalized Loewner
differential equation}.\medskip

The following strengthening of Theorem \ref{T1} leads to the method
of constructing quasiconformal extension, and is based on
the result due to Becker (see \cite{Be}, \cite{Be2} and also \cite%
{Be3}).

\begin{theorem}\cite[Becker]{Be, Be2, Be3}
\label{T11} Suppose that $\mathcal{L}(z,t)$ is a Loewner chain for which $%
p(z,t)$, defined  in \eqref{1}, satisfies the condition
\begin{eqnarray*}
p(z,t) &\in &U(k)\;:=\left\{ w\in \mathbb{C}
:\left\vert \;\frac{w-1}{w+1}\right\vert \leq k\right\} \\
&=&\left\{ w\in \mathbb{C}:\left\vert w-\;\frac{1+k^{2}}{1-k^{2}}\right\vert \leq \frac{2k}{1-k^{2}}%
\right\} \quad \left( 0\leq k<1\right)
\end{eqnarray*}%
for all $z\in \mathcal{U}$ and $t\in I$. Then $\mathcal{L}(z,t)$ admits a
continuous extension to $\overline{\mathcal{U}}$ for each $t\in I$ and the
function $F(z,\bar{z})$ defined by%
\begin{equation*}
F(z,\bar{z})=\left\{
\begin{array}{lcl}
\mathcal{L}(z,0) & for&{\left\vert z\right\vert <1}, \\
\mathcal{L}\left( \frac{z}{|z|},\log |z|\right) &for& {\left\vert
z\right\vert\geq 1},%
\end{array}%
\right.
\end{equation*}%
is a $k-$quasiconformal extension of $\mathcal{L}(z,0)$ to $\mathbb{C}$.
\end{theorem}

Detailed information about Loewner chains and quasiconformal
extension criterion can be found in \cite{Ah}, \cite{AnHi},
\cite{Bet}, \cite{ca}, \cite{Kr}, \cite{Pf}. For a  recent account
of the theory we refer the reader to \cite{Ho, Ho1, Ho2}.

\section{\textbf{Univalence criteria}}

The first theorem is our glimpse of the role of the generalized
Loewner chains  in univalence results for an operator $G_\alpha$,
studied in \cite{Ov1}. The theorem formulates the conditions under
which such an operator is analytic and univalent.

\begin{theorem}
\label{T2}\textit{\ Let }$\alpha ,\;c\;$and $s$\textit{\ be complex
numbers, that }$c\notin \left[ {0,\infty }\right) ;$ $s=a+ib,$
$a>0,$ $b\in \mathbb{R};$ $m>0$ and $f,g\in \mathcal{A}.$ If there
exists a function $h$,  analytic
in $\mathcal{U}$, and such that $h(0)=h_{0},$ $%
h_{0}\in \mathbb{C},$ $h_{0}\notin ({-\infty ,0]}$, and the
inequalities
\begin{equation}
\left\vert {{\alpha }-\frac{m}{2a}}\right\vert <\frac{m}{2a},  \label{eq1}
\end{equation}%
\begin{equation}
\left\vert {\frac{c}{h(z)}+\frac{m}{2{\alpha }}}\right\vert <\frac{m}{%
2\left\vert {\alpha }\right\vert },  \label{eq2}
\end{equation}%
and%
\begin{equation}
\left\vert {\frac{-c{\alpha }}{ah(z)}\left\vert z\right\vert ^{m/
a}+\left( {1-\left\vert z\right\vert ^{m/ a}}\right) }\left[
{(\alpha -1){\frac{zg'(z)}{g(z)}}}+{1+\frac{zf''(z)}{f'(z)}}+{{{\frac{zh'(z)}{h(z)}}}}\right] {-\frac{m}{2a}}%
\right\vert \le \frac{m}{2a}  \label{eq3}
\end{equation}%
hold true for all $z\in \mathcal{U},$ \textit{then the function }%
\begin{equation}
G_{\alpha }(z)=\left[ {\alpha \int\limits_{0}^{z}{{g^{\alpha
-1}(u)}f'(u)du}}\right] ^{1/ {\alpha }}  \label{eq4}
\end{equation}%
\textit{is analytic and univalent in }$\mathcal{U},$\textit{\ where the
principal branch is intended.}
\end{theorem}

\begin{proof}
We first prove that there exists a real number $r\in \left(
{0,1}\right] $ such that the function
$\mathcal{L}:\mathcal{U}_{r}\times I\rightarrow
\mathbb{C},$ defined formally by%
\begin{equation}
\mathcal{L}(z,t)=\left( {\alpha \int\limits_{0}^{e^{-st}z}{{g^{\alpha -1}(u)}%
f'(u)du}}{-\frac{a}{c}\left( {e^{mt}-1}\right) e^{-st}z{g^{\alpha
-1}(e^{-st}z)f'}(e^{-st}z)h(e^{-st}z)}\right) ^{1/ {\alpha }%
},  \label{eq5}
\end{equation}%
is analytic in $\mathcal{U}_{r}$ for all $t\in I.$

Because $g\in \mathcal{A},$ the function%
\begin{equation*}
\phi (z)=\frac{{g(z)}}{z}=1+...
\end{equation*}%
\noindent is analytic in $\mathcal{U}$ and $\phi (0)=1$. Then there
exist a disc $\mathcal{U}_{r_{1}},$ $0<r_{1}\le 1,$ in which $\phi
(z)\neq 0$ for all $z\in \mathcal{U}_{r_{1}}.$We denote by $\phi
_{1}$ the uniform branch of $\left( \phi {(z)}\right) ^{\alpha -1}$
equal to 1 at the origin.

Consider now a function
\begin{equation*}
\phi _{2}(z,t)={\alpha }\int\limits_{0}^{e^{-st}z}{u^{\alpha -1}\phi
}_{1}(u)f'(u)du{=e}^{-st{\alpha }}{z^{{\alpha }}+...,}
\end{equation*}%
\noindent then we have%
\begin{equation*}
\phi _{2}(z,t)=z^{{\alpha }}\phi _{3}(z,t)
\end{equation*}%
\noindent where $\phi _{3}$ is also analytic in $\mathcal{U}_{r_{1}}.$
Hence, the function%
\begin{equation*}
\phi _{4}(z,t)=\phi _{3}(z,t)-\frac{a}{c}\left( {e^{mt}-1}\right) e^{-st{%
\alpha }}\phi _{1}(e^{-st}z)f'(e^{-st}z)h(e^{-st}z)
\end{equation*}%
\noindent is analytic in $\mathcal{U}_{r_{1}}$ and%
\begin{equation*}
\phi _{4}(0,t)=e^{-st{\alpha }}\left[ {\left( {1+\frac{a}{c}}h_{0}\right) -%
\frac{a}{c}h_{0}e^{mt}}\right] .
\end{equation*}%
Now, we prove that $\phi _{4}(0,t)\neq 0$ for all $t\in I.$ It is
easy to
see that $\phi _{4}(0,0)=1.$ Suppose that there exists $t_{0}>0$ such that $%
\phi _{4}(0,t_{0})=0$. Then the equality $e^{mt_{0}}=\frac{c+ah_{0}%
}{ah_{0}}$ holds.  Since $h_{0}\notin ({-\infty ,0],}$ this equality implies that $%
c>0,$ which contradicts $c\notin \left[ {0,\infty }\right) .$ From
this we conclude
that $\phi _{4}(0,t)\neq 0$ for all $t\in I.$ Therefore, there is a disk $%
\mathcal{U}_{r_{2}},\;r_{2}\in \left( {0,r_{1}}\right] ,$ in which $\phi
_{4}(z,t)\neq 0$ for all $t\in I.$ Thus, we can choose an uniform branch of $%
\left[ \phi {_{4}(z)}\right] ^{1/ {\alpha }}$ analytic in $\mathcal{U}%
_{r_{2}},$ and denoted by $\phi _{5}(z,t).$

It follows from \eqref{eq5} that%
\begin{equation*}
\mathcal{L}(z,t)=z\phi _{5}(z,t)=a_{1}(t)z+a_{2}(t)z^{2}+...
\end{equation*}%
\noindent and consequently the function $\mathcal{L}(z,t)$ is analytic in $\mathcal{U%
}_{r_{2}}$.

We note that%
\begin{equation*}
a_{1}(t)=e^{t\left( {\frac{m}{{\alpha }}-s}\right) }\left[ {\left( {1+\frac{a%
}{c}}h_{0}\right) e^{-mt}-\frac{a}{c}}h_{0}\right] ^{1/ {\alpha }},
\end{equation*}
for which we consider the uniform branch equal to $1$ at the origin.

Since $\left\vert {a\alpha -\frac{m}{2}}\right\vert <\frac{m}{2}$ is
equivalent to $\Re \left\{ {\frac{m}{{\alpha }}}\right\} >a=\Re
(s),$ we
have that%
\begin{equation*}
\underset{t\rightarrow \infty }{\lim }\left\vert {a_{1}(t)}\right\vert
=\infty .
\end{equation*}%
Moreover, $a_{1}(t)\neq 0$ for all $t\in I.$

From the analyticity of $\mathcal{L}(z,t)$ in $\mathcal{U}_{r_{2}},$
it follows that there exists a number $r_{3}$ such that
$0<r_{3}<r_{2},$ and a
constant $K=K(r_{3})$ such that%
\begin{equation*}
\left\vert {\frac{\mathcal{L}(z,t)}{a_{1}(t)}}\right\vert <K\quad (
z\in \mathcal{U}_{r_{3}},\ t\in I).
\end{equation*}%
By the Montel's Theorem, $\left\{ {\frac{\mathcal{L}(z,t)}{a_{1}(t)}}%
\right\} _{t\in I}$\ forms a normal family in $\mathcal{U}_{r_{3}}.$
From the analyticity of $\frac{\partial \mathcal{L}(z,t)}{\partial
t},$ it may be concluded  that for all fixed numbers $T>0$ and
$r_{4},\;0<r_{4}<r_{3},$ there exists a
constant $K_{1}>0$ (that depends on $T$ and $r_{4})$ such that%
\begin{equation*}
\left\vert {\frac{\partial \mathcal{L}(z,t)}{\partial t}}\right\vert
<K_{1}\quad (z\in \mathcal{U}_{r_{4}},\ t\in \left[ {0,T}\right]).
\end{equation*}%
Therefore, the function $\mathcal{L}(z,t)$ is locally absolutely continuous
in $I,$ locally uniform with respect to $\mathcal{U}_{r_{4}}.$

Let $p:\mathcal{U}_{r}\times I\rightarrow \mathbb{C}$ denote a function%
\begin{equation*}
p(z,t)={z\frac{\partial \mathcal{L}(z,t)}{\partial z}}/
\frac{\partial \mathcal{L}(z,t)}{\partial t},
\end{equation*} that is analytic  in $\mathcal{U}_{r},\;0<r<r_{4},$ for all $t\in
I$.
If the function%
\begin{equation}
w(z,t)=\frac{p(z,t)-1}{p(z,t)+1}=\frac{\frac{z\partial \mathcal{L}(z,t)}{%
\partial z}-\frac{\partial \mathcal{L}(z,t)}{\partial t}}{\frac{z\partial
\mathcal{L}(z,t)}{\partial z}+\frac{\partial \mathcal{L}(z,t)}{\partial t}}
\label{eq6}
\end{equation}%
\noindent is analytic in $\mathcal{U}\times I$, and $\left\vert {w(z,t)}%
\right\vert <1,$ for all $z\in \mathcal{U}\;$and $t\in I,$ then $p(z,t)$ has
an analytic extension with positive real part in $\mathcal{U},$ for all $%
t\in I.$ According to \eqref{eq6} we have%
\begin{equation}
w(z,t)=\frac{(1+s)A(z,t)-m}{(1-s)A(z,t)+m},  \label{eq7}
\end{equation}%
\noindent where%
\begin{eqnarray}
A(z,t) &=&\frac{-c{\alpha }}{ah(e^{-st}z)}e^{-mt}+\left( {1-e^{-mt}}\right) %
\left[ {(\alpha -1){\frac{e^{-st}zg'(e^{-st}z)}{g(e^{-st}z)}}}%
\right.  \label{eq8} \\
&&\left.
{{+}1+\frac{e^{-st}zf''(e^{-st}z)}f'(e^{-st}z)}+\frac{e^{-st}zh'(e^{-st}z)}{h(e^{-st}z)}\right]
\notag
\end{eqnarray}%
for $z\in \mathcal{U}$ and $t\in I.$ Hence, the inequality $|w(z,t)| <1$  is equivalent to%
\begin{equation}
\left\vert {A(z,t)-\frac{m}{2a}}\right\vert <\frac{m}{2a},\text{
}a=\Re (s) \quad (z\in \mathcal{U},\ t\in I).  \label{eq9}
\end{equation}%
Define now%
\begin{equation*}
B(z,t)=A(z,t)-\frac{m}{2a}\quad (z\in \mathcal{U},\ t\in I).
\end{equation*}%
From \eqref{eq1}, \eqref{eq2} and \eqref{eq8} it follows that%
\begin{equation}
\left\vert {B(z,0)}\right\vert =\left\vert {\frac{c{\alpha }}{ah(z)}+\frac{m%
}{2a}}\right\vert <\frac{m}{2a},  \label{eq10}
\end{equation}%
\noindent and%
\begin{eqnarray}
\left\vert {B(0,t)}\right\vert &=&\frac{1}{a}\left\vert \frac{{c\alpha
e^{-mt}}}{h_{0}}{-a\alpha \left( {1-e^{-mt}}\right) +\frac{m}{2}}\right\vert
\label{eq11} \\
&=&\frac{1}{a}\left\vert {\left( \frac{{c\alpha }}{h_{0}}{+\frac{m}{2}}%
\right) e^{-mt}+\left( {\frac{m}{2}-a\alpha }\right) \left( {1-e^{-mt}}%
\right) }\right\vert\\ & <&\frac{m}{2a}.  \notag
\end{eqnarray}%
Since $\left\vert {e^{-st}z}\right\vert \le \left\vert {e^{-st}}%
\right\vert =e^{-at}<1$ for all $z\in \overline{\mathcal{U}}=\left\{
{z\in \mathbb{C}:\;\|z| \le 1}\right\} $ and $t>0,$ we
conclude that for each $t>0$ $B(z,t)$ is an analytic function in $\overline{%
\mathcal{U}}.$ Using the maximum modulus principle it follows that for all $%
z\in \mathcal{U}\setminus\{0\}$ and each $t>0$ arbitrarily fixed there exists $%
\theta =\theta (t)\in \mathbb{R}$ such that%
\begin{equation}
\left\vert {B(z,t)}\right\vert <\underset{\left\vert z\right\vert =1}{\lim }%
\left\vert {B(z,t)}\right\vert =\left\vert {B(e^{i\theta },t)}\right\vert ,
\label{eq12}
\end{equation}%
\noindent for all $z\in \mathcal{U}\;$and $t\in I.$

Denote $u=e^{-st}e^{i\theta }.$ Then $\left\vert u\right\vert
=e^{-at}$, and from \eqref{eq8} we obtain%
\begin{eqnarray*}
\left\vert {B(e^{i\theta },t)}\right\vert &=&\left\vert {\frac{c{\alpha }}{%
ah(u)}|u| ^{m/ a}+\frac{m}{2a}}\right. \\
&&\left. {-\left( {1-\left\vert u\right\vert ^{m/ a}}\right) }\left[ {{%
(\alpha
-1){\frac{ug'(u)}{g(u)}}}}{{+{1+\frac{uf''(u)}{f'(u)}}+{\frac{uh'(u)}{h(u)}}}}\right]
\right\vert.
\end{eqnarray*}%
Since $u\in \mathcal{U},$ the inequality \eqref{eq3} implies that%
\begin{equation*}
\left\vert {B(e^{i\theta },t)}\right\vert \le \frac{m}{2a},
\end{equation*}%
and from \eqref{eq10}, \eqref{eq11} and \eqref{eq12}, we conclude that%
\begin{equation*}
\left\vert {B(z,t)}\right\vert =\left\vert {A(z,t)-\frac{m}{2a}}\right\vert <%
\frac{m}{2a}
\end{equation*}%
\noindent for all $z\in \mathcal{U}\;$and $t\in I.$ Therefore $\left\vert {%
w(z,t)}\right\vert <1$ for all $z\in \mathcal{U}\;$and $t\in I.$

Since all the conditions of Theorem 2.1 are satisfied, we obtain that the
function $\mathcal{L}(z,t)$ has an analytic and univalent extension to the
whole unit disk $\mathcal{U},$ for all $t\in I.$ For $t=0$ we have $\mathcal{%
L}(z,0)=G_{\alpha }(z),$ for $z\in \mathcal{U}$ and therefore the function $%
G_{\alpha }(z)$ is analytic and univalent in $\mathcal{U}.$
\end{proof}

Abbreviating \eqref{eq3} we can now rephrase Theorem \ref{T2} in a
simpler form.

\begin{theorem}
\label{T21}\textit{\ Let } $f,g\in \mathcal{A}.$ \textit{Let }$m>0$,
the complex numbers $\alpha ,\;c,$ $s$ and the function $h$\textit{\
be as in }Theorem \ref{T2}. Moreover, suppose that the inequalities
\eqref{eq1} and \eqref{eq2} are satisfied. If the inequality
\begin{equation}
\left\vert(\alpha-1)\frac{zg'(z)}{g(z)}+1+\frac{zf''(z)}{f'(z)}+\frac{zh'(z)}{h(z)}-
\frac{m}{2a}\right\vert \le \frac{m}{2a}  \label{eq211}
\end{equation}%
holds true for all  $z\in \mathcal{U},$ then the function $%
G_{\alpha }$ defined by  \eqref{eq4}  is analytic and univalent in
$\mathcal{U}.$
\end{theorem}

\begin{proof}
Making use of \eqref{eq2} and \eqref{eq211} we obtain%
\begin{eqnarray*}
&&\left\vert {\frac{c{\alpha }}{h(z)}\left\vert z\right\vert ^{m/ a}+%
\frac{m}{2}-a\left( {1-\left\vert z\right\vert ^{m/ a}}\right)
}\left[ {(\alpha
-1){\frac{zg'(z)}{g(z)}}}+{1+\frac{zf''(z)}{f'(z)}}+{{{\frac{zh'(z)}{h(z)}}}}\right]
\right\vert \\
&=&\left\vert \left( {\frac{c{\alpha }}{h(z)}+\frac{m}{2}}\right) {%
\left\vert z\right\vert ^{m/ a}}\right. \\
&&\left. +{\left( {1-\left\vert z\right\vert ^{m/ a}}\right) }\left[
-a\left( {(\alpha -1){\frac{zg'(z)}{g(z)}}}+{1+\frac{zf''(z)}{f'(z)}}+{{{\frac{zh'(z)}{h(z)}}}}\right) {+%
\frac{m}{2}}\right] \right\vert \\
&\leq &{\left\vert z\right\vert ^{m/ a}\frac{m}{2}+\left( {%
1-\left\vert z\right\vert ^{m/ a}}\right) \frac{m}{2}=}\frac{m}{2},
\end{eqnarray*}%
so that the condition \eqref{eq3} is satisfied. This finishes the
proof, since all the assumption of Theorem \ref{T2} are satisfied.
\end{proof}

The special case of Theorem \ref{T2} i.e. for $s=\alpha =1$, and
$h(z)=-c$ leads to the following result.

\begin{corollary}
\label{Ce1} \textit{Let } $f\in \mathcal{A}$ and $m>1.$ If%
\begin{equation}
\left\vert \frac{m-2}{2}-{\left( {1-\left\vert z\right\vert
^{m}}\right) \frac{zf''(z)}{f'(z)}}\right\vert \le \frac{m}{2}
\label{e4}
\end{equation}%
holds for $z\in \mathcal{U}$, then the function $f$\ %
univalent in $\mathcal{U}$.
\end{corollary}

Corollary \ref{Ce1} in turn implies the well-known Becker's
univalence citerion \cite{Be}.

\begin{remark}
\label{R1} Important examples of univalence criteria may be obtained
by a suitable choices of $f$ and $g$, below.

\begin{enumerate}
\item Choose $g_{1}(z)=z$. Then Theorem \ref{T2} gives analyticity
and univalence of the operator
\begin{equation*}
F(z)=\left[ {\alpha \int\limits_{0}^{z}{u^{\alpha -1}f'(u)du}}%
\right] ^{1/ \alpha },
\end{equation*}%
which was studied by Pascu \cite{Pa}.

\item Setting $f(z)=z$ in Theorem \ref{T2}, we obtain that the operator%
\begin{equation*}
G(z)=\left[ {\alpha \int\limits_{0}^{z}{g^{\alpha -1}(u)du}}\right]
^{1/ \alpha }
\end{equation*}%
is analytic and univalent in $\mathcal{U}.$ The operator $G$ was
introduced by Moldoveanu and Pascu \cite{MoPa1}.

\item Taking $f'(z)=\frac{g(z)}{z}$ in Theorem \ref{T2} we
find that%
\begin{equation*}
H(z)=\left[ {\alpha \int\limits_{0}^{z}{\frac{g^{\alpha
}(u)}{u}du}}\right] ^{1/ \alpha }
\end{equation*}%
is analytic and univalent in $\mathcal{U}$. The operator $H$ was
introduced and studied by Mocanu \cite{Mo}.
\end{enumerate}
\end{remark}

If we limit a range of parameter $a$ to the case $a\ge 1$ then,
applying the Theorem \ref{T2}, we obtain the following.

\begin{theorem}
\label{T3}  Let $\alpha ,\;c\;$and $s$  be complex numbers, that
$c\notin \left[ {0,\infty }\right) ;$ $s=a+ib,$ $a\geq 1,$ $b\in
\mathbb{R};$ $m>0$ and $f,g\in \mathcal{A}.$  Let the function $h$ \
be as in Theorem \ref{T2}. Moreover, suppose that
the inequalities \eqref{eq1} and \eqref{eq2} are satisfied. If the inequality%
\begin{equation}
\left\vert {\frac{-c\alpha }{ah(z)}\left\vert z\right\vert ^{m}+\left( {%
1-\left\vert z\right\vert ^{m}}\right) }\left[ {(\alpha -1){\frac{zg'(z)}{g(z)}}}+{1+\frac{zf''(z)}{f'(z)}}+%
{{{\frac{zh'(z)}{h(z)}}}}\right] {-\frac{m}{2a}}\right\vert \le
\frac{m}{2a}  \label{eq18}
\end{equation}%
holds true for all $z\in \mathcal{U},$ then the function $%
G_{\alpha }(z)$ defined by \eqref{eq4}  is analytic and univalent in
$\mathcal{U}.$
\end{theorem}

\begin{proof}
For $\lambda \in \left[ {0,1}\right]$ define the linear function%
\begin{equation*}
\phi (z,\lambda )=\lambda k(z)+\left( {1-\lambda }\right) l(z),\text{ \ \ }%
(z\in \mathcal{U},\;t\in I)
\end{equation*}%
\noindent where%
\begin{equation*}
k(z)=\frac{c\alpha }{h(z)}+\frac{m}{2}
\end{equation*}%
\noindent and%
\begin{equation*}
l(z)=-a\left[ {(\alpha -1){\frac{zg'(z)}{g(z)}}}+{1+\frac{zf''(z)}{f'(z)}}+{{\frac{zh'(z)}{h(z)}}}%
\right] +\frac{m}{2}.
\end{equation*}%
For fixed $z\in \mathcal{U}$ and $t\in I,$ $\phi (z,\lambda )$ is a
point of a segment with endpoints at $k(z)$ and $l(z).$ The function
$\phi (z,\lambda )$ is analytic in $\mathcal{U}$ for all $\lambda
\in \left[ {0,1}\right]$ and $z\in \mathcal{U},$ satisfies
\begin{equation}
|phi (z,1)| =|k(z)| <\frac{m}{2} \label{eq19}
\end{equation}%
\noindent and%
\begin{equation}
\left\vert {\phi (z,\left\vert z\right\vert ^{m})}\right\vert \le
\frac{m}{2},  \label{eq20}
\end{equation} which follows from \eqref{eq2} and \eqref{eq18}.
If $\lambda $ increases from $\lambda _{1}=\left\vert z\right\vert ^{m}$ to $%
\lambda _{2}=1,$ then the point $\phi (z,\lambda )$ moves on the segment
whose endpoints are $\phi (z,\left\vert z\right\vert ^{m})$ and $\phi (z,1).$
Because $a\ge 1,$ from \eqref{eq19} and \eqref{eq20} it follows that%
\begin{equation}
\left\vert {\phi (z,\left\vert z\right\vert ^{m/ a})}\right\vert \le
\frac{m}{2},\quad z\in \mathcal{U}. \label{eq21}
\end{equation}%
We can observe that the inequality \eqref{eq21} is just the
condition \eqref{eq3}, and then  Theorem \ref{T2} now yields that
the function $G_{\alpha }(z)$, defined by \eqref{eq4}, is analytic
and univalent in $\mathcal{U}.$
\end{proof}

\begin{remark}\label{R4}
Applying Theorem \ref{T3} to $m=2$ and the function $h(z)\equiv 1$,
and  $g(z)=f(z), \alpha =1/ s$ (or $g(z)=z,$  $a=1,$
$c=-\frac{1}{\alpha }$, respectively) we obtain the results by
Ruscheweyh \cite{Ru}  (or Moldoveanu and Pascu \cite{MoPa2},
resp.).\end{remark}

\begin{remark}
\label{R5}  Substituting $1/ h$ instead of $h$ with $h(0)=1$ and
setting $g(z)=f(z), \;\alpha =1/ s, m=2$  in Theorem \ref{T3} we
obtain the result due to Singh and Chichra \cite{SiCh}.
\end{remark}

\begin{remark}
\label{R6}  Setting $g(z)=f(z),\;s=\alpha =1,\;c=-1,\;m=2$ and
$h(z)=\frac{k(z)+1}{2},$ where $k$ is an analytic function with
positive real part in $\mathcal{U}$ with $k(0)=1$ in Theorem
\ref{T3}, we obtain the result  by Lewandowski
\cite{Le}.\end{remark}

\begin{remark}
\label{R7} For the case when $m=2$ and $h(0)=h_{0}=1$ Theorem
\ref{T2} and \ref{T3} reduce to the results by Ovesea \cite{Ov1}.
\end{remark}

\section{\textbf{Quasiconformal extension criterion}}

In this section we will refine the univalence condition given in Theorem \ref%
{T2} to a quasiconformal extension criterion.

\begin{theorem}
\label{T5}  Let $\alpha ,\;c\;$and $s$  be complex numbers, that
$c\notin \left[ {0,\infty }\right) ;$ $s=a+ib,$ $a>0,$ $b\in
\mathbb{R} ;$ $m>0;$ $k\in \lbrack 0,1)$ and let $f,g\in
\mathcal{A}.$ If there exists a function $h$,  analytic
in $\mathcal{U}$, such that $h(0)=h_{0},$ $h_{0}\in \mathbb{C},$ $h_{0}\notin ({-\infty ,0]}$ and the inequalities %
\begin{equation*}
\left\vert {\alpha -\frac{m}{2a}}\right\vert <\frac{m}{2a},
\end{equation*}%
\begin{equation}
\left\vert {\frac{c{\alpha }}{h(z)}+\frac{m}{2}}\right\vert <k\frac{m}{2}
\label{4.1}
\end{equation}%
\textit{and}%
\begin{equation}
\left\vert {\frac{-c{\alpha }}{ah(z)}\left\vert z\right\vert ^{m/
a}+\left( {1-\left\vert z\right\vert ^{m/ a}}\right) }\left[
{(\alpha
-1){\frac{zg'(z)}{g(z)}}}+{1+\frac{zf''(z)}{f'(z)}}+{{{\frac{zh'(z)}{h(z)}}}}\right] {-\frac{m}{2a}}%
\right\vert \le k\frac{m}{2a}  \label{4.112}
\end{equation}%
hold true for all $z\in \mathcal{U},$ then the function $%
G_{\alpha }(z)$ given by \eqref{eq4} has an $K-$quasiconformal extension to $%
\mathbb{C},$ where
\begin{equation}
K=\left\{\begin{array}{lcl}
k & for & s=1, \\
\dfrac{| s-1| ^{2}+k|\bar{s}^{2}-1|}{|\bar{s}^{2}-1| +k|s-1|^{2}} &
for & s\neq 1.
\end{array}
\right.  \label{4.111}
\end{equation}
\end{theorem}

\begin{proof}
In the proof of Theorem \ref{T2} it has been shown  that the function $\mathcal{L%
}(z,t)$, given by \eqref{eq5}, is a subordination chain in
$\mathcal{U}$. Applying Theorem \ref{T11} to the function $w(z,t)$
given by \eqref{eq7}, we obtain that the condition
\begin{equation}
\left\vert \frac{(1+s)A(z,t)-m}{(1-s)A(z,t)+m}\right\vert <l\qquad
(z\in \mathcal{U},\ t\in I,\ 0\le l< 1) \label{4.2}
\end{equation}%
with $A(z,t)$  defined by \eqref{eq8}, implies $l-$quasiconformal
extensibility of $G_{\alpha }(z)$. Lengthy, but elementary
calculations, show that the inequality \eqref{4.2} is equivalent to
\begin{equation}
\left| A(z,t)-\frac{m\left( (1+l^{2})+a(1-l^{2})-ib(1-l^{2})\right) }{%
2a(1+l^{2})+(1-l^{2})(1+|s|^{2})}\right| \leq \frac{%
2lm}{2a(1+l^{2})+(1-l^{2})(1+|s|^{2})}. \label{4.3}
\end{equation}
Taking  into account \eqref{4.1} and \eqref{4.112}, we clearly see
that
\begin{equation}
\left\vert A(z,t)-\frac{m}{2a}\right\vert \leq k\frac{m}{2a}.  \label{4.4}
\end{equation}
Consider the two disks $\Delta_1(s_1,r_1)$ and $\Delta_2(s_2,r_2)$
defined by \eqref{4.3} and \eqref{4.4} respectively, where $A(z,t)$
is replaced by a complex variable $w$. The proof is completed by
showing that there exists $l\in \lbrack 0,1)$ for which $\Delta_2$
is contained in $\Delta_1$. Equivalently $\Delta_2\subset \Delta_1$
holds, if $|s_1-s_2|+r_2\le r_1$, that is
\begin{equation*}
\left\vert \frac{m\left( (1+l^{2})+a(1-l^{2})\right) -imb(1-l^{2})}{%
2a(1+l^{2})+(1-l^{2})(1+\left\vert s\right\vert ^{2})}-\frac{m}{2a}%
\right\vert +k\frac{m}{2a}\leq \frac{2lm}{2a(1+l^{2})+(1-l^{2})(1+\left\vert
s\right\vert ^{2})}
\end{equation*}%
or
\begin{equation}
\frac{(1-l^{2})|\bar{s}^{2}-1|}{2a\left[
2a(1+l^{2})+(1-l^{2})(1+|s|^{2})\right] }\leq \frac{2l}{%
2a(1+l^{2})+(1-l^{2})(1+|s|^{2})}-\frac{k}{2a} \label{4.5}
\end{equation}%
with the condition
\begin{equation}
\frac{2l}{2a(1+l^{2})+(1-l^{2})(1+\left\vert s\right\vert ^{2})}-\frac{k}{2a}%
\geq 0.  \label{4.6}
\end{equation}%
For the case,  when $k=0$, the condition \eqref{4.6} holds for every
$l$, while \eqref{4.5} is satisfied for $l_1\le l<1$, where
\begin{equation}\label{4.7}
l_1=\frac{|s-1|^2}{|\bar{s}^2-1|}.
\end{equation}
If, on the other hand, $s=1$ and $k\in (0,1)$, then \eqref{4.6} and
\eqref{4.5} hold for $k\le l<1$. Assume now $s\neq 1$, and $k\in
(0,1)$. The condition \eqref{4.6} reduces to the quadratic
inequality
$$l^2\left[k(1+|s|^2)-2ak\right]+4al-k[2a+1+|s|^2]\ge 0,$$ or
\begin{equation}\label{4.8}
kl^2|s-1|^2+4al-k|s+1|^2\ge 0.
\end{equation}
Therefore, we find that \eqref{4.6} (or \eqref{4.8}) holds for
$l_2\le l<1$, where
\begin{equation}\label{4.9}
l_2=\frac{\sqrt{4a^2+k^2|\bar{s}^2-1|^2}-2a}{k|s-1|^2}.
\end{equation}
Similarly, \eqref{4.6} may be rewritten as
$$(1-l^2)|\bar{s}^2-1|\le 4al-2ak(1+l^2)-k(1-l^2)(1+|s|^2),$$ or
\begin{equation}\label{4.10}
 l^2\left[k|s-1|^2+|\bar{s}^2-1|\right]+4al-k|s+1|^2-|\bar{s}^2-1|\ge
 0,
\end{equation}
that is satisfied for $l_3\le l<1$, where
\begin{equation}\label{4.11}
l_3=\frac{| s-1| ^{2}+k|\bar{s}^{2}-1|}{|\bar{s}^{2}-1|
+k|s-1|^{2}}.
\end{equation}
We note that $l_2\le l_3$. Indeed, it is trivial that
$$\left[|\bar{s}^{2}-1|+k|s-1|^{2}\right]\sqrt{4a^2+k^2|\bar{s}^{2}-1|^2}\le
\left[|\bar{s}^{2}-1|+k|s-1|^{2}\right]\left[2a+k
|\bar{s}^2-1|\right].$$ Moreover, we see at once that
$$\left[|\bar{s}^{2}-1|+k|s-1|^{2}\right]\left[2a+k
|\bar{s}^{2}-1|\right]\le
\left[|\bar{s}^{2}-1|+k|s-1|^{2}\right]\left[2a+k
|\bar{s}^2-1|\right]+4ak|s-1|^2.$$ Combining the last two
inequalities we obtain
$$\left[|\bar{s}^{2}-1|+k|s-1|^{2}\right]\sqrt{4a^2+k^2|\bar{s}^{2}-1|^2}\le
\left[|\bar{s}^{2}-1|+k|s-1|^{2}\right]\left[2a+k
|\bar{s}^{2}-1|\right]+4ak|s-1|^2,$$ which is equivalent to the
desired inequality $l_2\le l_3$.  Likewise, it is a simple matter to
show that $l_3<1$, and the proof is complete, by setting $K:=l_3$.
We note also, that the case $k=0$ may be included to the last case
(i.e. $s\neq 1$).
\end{proof}

Several similar sufficient conditions for quasiconformal extensions
as in the Theorem \ref{T5} can be derived. Here we select a few
example out of a large variety of possibilities. The following is
based on the Theorem \ref{T11}.

\begin{theorem}
\label{T6}Let $\alpha >0$, and $f,g\in \mathcal{A}$. If%
\begin{equation*}
z^{1-\alpha }g(z)^{\alpha -1}f'(z)\in U(k)
\end{equation*}%
for all $z\in \mathcal{U},$ then the function $G_{\alpha
}(z)$ can be extended to a $\ k-$quasiconformal automorphism of $%
\mathbb{C}.$
\end{theorem}

\begin{proof}
Set%
\begin{equation*}
\mathcal{L}(z,t)=\left({\alpha \int\limits_{0}^{z}{{g^{\alpha -1}(u)}f'(u)du}+({e^{\alpha t}-1)}}z{{{^{\alpha }}}}\right) ^{1/ {%
\alpha }}.
\end{equation*}%
An easy computation shows%
\begin{equation*}
p(z,t)=\frac{1}{{{e^{\alpha t}}}}\left( z^{1-\alpha }g(z)^{\alpha
-1}f'(z)\right) +\left( 1-\frac{1}{{{e^{\alpha t}}}}\right),
\end{equation*}%
and the assertion follows by the same methods as in Theorem
\ref{T5}, applying Theorems \ref{T1} and \ref{T11}.
\end{proof}

In the same manner, by definition of the suitable Loewner chain,
several univalence criterion may by found. For example,  the condition%
\begin{equation*}
\frac{zG_{\alpha }'(z)}{G_{\alpha }(z)}\in U(k)\quad (\alpha \in
\mathbb{C}),
\end{equation*}%
which is based on the integral operator $G_{\alpha }(z)$, is given
by the
Loewner chain%
\begin{equation*}
\mathcal{L}(z,t)={{e^{t}}}G_{\alpha }(z).
\end{equation*}

\bigskip
\noindent This work was partially supported by the Centre for
Innovation and Transfer of Natural Sciences and Engineering
Knowledge, Faculty of Mathematics and Natural Sciences, University
of Rzeszow.\\

\end{document}